\documentclass{article}
\usepackage{graphicx} 
\usepackage[left=1in,top=1in,right=1in,bottom=1in]{geometry}
\usepackage[linktocpage=true]{hyperref}
\usepackage{setspace}
\usepackage{tikz}
\usetikzlibrary{decorations.pathreplacing}
\usetikzlibrary{calc}
\usepackage{amssymb, amsmath, amsthm, graphicx,mathrsfs}
\usepackage{caption,cite}
\usepackage{subcaption}
\usepackage{mathtools,color}

\usepackage{cleveref}
\newcommand{\p}{\mathbb{P}}
\newcommand{\e}{\mathbb{E}}

\usepackage{todonotes}

\newcommand{\E}{\mathbb{E}}

\renewcommand{\P}{\mathbb{P}}

\newtheorem{theorem}{Theorem}[section]
\newtheorem{lemma}[theorem]{Lemma}
\newtheorem{corollary}[theorem]{Corollary}

\newtheorem{conjecture}[theorem]{Conjecture}

\newtheorem*{theorem*}{Theorem}

\tikzset{vtx/.style={inner sep=1.7pt, outer sep=0pt, circle, fill=black,draw}}

\title{Improving {$R(3,k)$} in just two bites}
\author{ Zion Hefty\thanks{Department of Mathematics, University of Denver, Denver, USA. Email: {\tt Zion.Hefty@du.edu}}
\and  Paul Horn\thanks{Department of Mathematics, University of Denver, Denver, USA and Department of Mathematics and Applied Mathematics, University of Johannesburg, Johannesburg, South Africa. Email: {\tt Paul.Horn@du.edu}} 
\and Dylan King\thanks{California Institute of Technology, Pasadena, USA. Email: {\tt dking@caltech.edu}} 
\and Florian Pfender\thanks{Department of Mathematical and Statistical Sciences, University of Colorado Denver, Denver, USA. E-mail: {\tt Florian.Pfender@ucdenver.edu}. Research is partially supported by NSF FRG DMS-2152498.}}
\date{\today}

\begin{document}

\maketitle

\begin{abstract}
We present a flexible random construction which, for certain graphs~$H$, is able to produce~$H$-free graphs with edge density strictly larger than that of the~$H$-free process, while simultaneously preserving pseudorandom properties and allowing a much easier analysis.

As our main application, we use this construction to show that the off-diagonal Ramsey numbers satisfy $R(3,k)\ge  \left(\frac12+o(1)\right)\frac{k^2}{\log{k}}$, improving the previously best bound $R(3,k)\ge  \left(\frac13+o(1)\right)\frac{k^2}{\log{k}}$. 
While the best known upper bound is~$R(3,k)\le  \left(1+o(1)\right)\frac{k^2}{\log{k}}$, the constant of~$\frac12$ has been conjectured to be asymptotically tight by multiple groups.
\end{abstract}

\section{Introduction}
The Ramsey number $R(\ell,k)$ is the least~$N$ such that every~$N$-vertex graph contains either a clique of size~$\ell$ or an independent set of size~$k$.
In 1930 Ramsey~\cite{Ramsey1930} showed the existence of~$R(\ell,k)$, and in the intervening century much effort has been dedicated to understanding the asymptotic behavior of these numbers.
The best current bounds on the diagonal Ramsey numbers $R(k,k)$ (which are known exactly only up to $k=4$, with $43\le R(5,5)\le 46$ and the upper bound only very recently announced by Angeltveit and McKay~\cite{angeltveit2025}) are
\[
2^{k/2+o(1)}\le R(k,k)\le (4-c)^k,
\]
for some $c>0$.
The lower bound is one of the earliest instances of Erd\H{o}s' probabilistic method~\cite{Erdos47}, and the upper bound was the first exponential improvement over a result of Erd\H{o}s and Szekeres~\cite{ES35}, only recently announced in 2023 by Campos, Griffiths, Morris and Sahasrabudhe~\cite{campos2025exponentialimprovementdiagonalramsey}, and subsequently optimized to a value of $c\approx 0.2$ (for sufficiently large~$k$) by Gupta, Ndiaye, Norin and Wei~\cite{gupta2024optimizingcgmsupperbound}.

Perhaps the second most studied family of  Ramsey numbers are the extreme off-diagonal numbers $R(\ell,k)$ for~$\ell$ fixed and~$k$ tending to infinity, where the case $\ell=3$ yields the first non-trivial asymptotic question. 
The exciting history of the study of $R(3,k)$ until the early 2000s is described by Spencer~\cite{Spencer2011} in a chapter of the book ``Ramsey Theory"~\cite{SoiferRamsey}, while the more recent history is very well captured by 
Campos, Jenssen, Michelen and Sahasrabudhe~\cite{campos2025}. We briefly summarize the results determining the order of magnitude.
The upper bound of $R(3,k) = O(\frac{k^2}{\log{k}})$ goes back to work of Ajtai, Koml\'os and Szeme\'redi~ \cite{AjtaiKSz80,AjtaiKSz81} in the early 1980s, and in 1983 Shearer~\cite{Shearer83} reduced the implied constant to~$R(3,k) \leq (1+o(1))\frac{k^2}{\log{k}}$.
The lower bound was improved in several steps until Kim~\cite{Kim95} proved in 1995 that~$R(3,k)\ge c\frac{k^2}{\log{k}}$, with a constant $c\approx\frac1{160}$. Kim's construction uses a probabilistic method commonly known as a ``nibble'', where small batches of edges are added randomly for many steps.
Doing so carefully, one can avoid the appearance of triangles while keeping the edge distribution sufficiently random to destroy large independent sets.   
Spencer~\cite{Spencer2011} wrote in 2011:

\vspace{0.2cm}
{\em 
``The value
of a constant $c$ so that $R(3,k)\sim
c\frac{k^2}{\log{k}}$ remains open to this day, but this problem
seems beyond our reach."
}

\vspace{0.2cm}
Nevertheless, only a short time later in 2013, two groups announced that they had greatly improved on the constant in the lower bound, showing that~$R(3,k)\ge \left( \frac14 +o(1)\right)\frac{k^2}{\log{k}}$. Bohman and
Keevash~\cite{BohmanKeevash21} and Fiz Pontiveros, Griffiths and Morris~\cite{Fiz20} very carefully analyzed the triangle-free process, in which an edge, chosen uniformly at random from those which would not create a triangle, is added to the graph (and this operation is repeated until no such edges exist). 
Both of these papers go through serious technical efforts to control the randomness in the process, and in particular show
that with high probability the triangle-free process terminates with independence number yielding this bound. 
Fiz Pontiveros, Griffiths and Morris went so far as to conjecture that $c=\frac14$ is the correct constant for $R(3,k)$.

In another breakthrough 12 years later, Campos, Jenssen, Michelen and Sahasrabudhe~\cite{campos2025} showed that the story of the lower bound was not finished, and in fact $R(3,k)\ge \left( \frac13 +o(1)\right)\frac{k^2}{\log{k}}$. They conjecture that $c=\frac12$ should be the correct constant. 
\begin{conjecture}[Campos, Jenssen, Michelen and Sahasrabudhe~\cite{campos2025}]\label{con1/2}
\[R(3,k)=  \left(\frac12+o(1)\right)\frac{k^2}{\log{k}}.\]
\end{conjecture}

A conjecture by Davies, Jenssen, Perkins, and Roberts~\cite{Davies18} relating maximum and average size of independence sets in triangle-free graphs would imply the upper bound in Conjecture~\ref{con1/2}. In this paper, we prove the lower bound. 

\begin{theorem}\label{main2}
    \[
R(3,k)\ge  \left(\frac12+o(1)\right)\frac{k^2}{\log{k}}.
\]
\end{theorem}

Similarly to all other previous proofs of lower bounds, we prove the following theorem concerning independence numbers, which easily implies Theorem~\ref{main2}.

\begin{theorem}\label{main}
    For all $\varepsilon > 0$, there exists an $n_0=n_0(\varepsilon)$ so that for $n \geq n_0$ there exists a triangle-free graph $G$ on $n$ vertices with independence number 
    \[
\alpha(G) < (1+\varepsilon)\sqrt{n \log n}.
\]
\end{theorem}

\noindent Spencer states in the epilogue of his book chapter:

\vspace{0.2cm}
{\em 
``Is the story of $R(3,k)$ over? I think not. I think there is plenty of room for a consolidation of the results. My dream is a ten-page paper which gives $R(3,k)=\Theta(k^2/\log{k})$."
}

\vspace{0.2cm}
While inspired by ideas in~\cite{campos2025}, our proof of Theorem~\ref{main} does not require a nibble, and thus we are able to circumvent many of the technical difficulties faced by previous groups.
If one were interested only in the correct order of magnitude, our construction could achieve Spencer's dream, by combining slightly weaker ingredients with Shearer's one page proof of the upper bound.

In Section \ref{sec:construction} we present our random construction, which is quite versatile and has already seen applications to related problems. In Section~\ref{sec:hyper} we  give a direct application to hypergraph Ramsey numbers of stars, suggested to us by Mubayi.  Campos, Jenssen, Michelen and Sahasrabudhe together with the fourth author~\cite{campos2025polynomialimprovementoddcyclecomplete} have built on our method to find the first polynomial improvements over the edge deletion threshold 
for the cycle-complete Ramsey numbers $r(C_{\ell},K_k)$, for all odd $\ell\ge 9$ as~$k$ tends to infinity. This shows some flexibility in the choice of an excluded graph $H$ beyond triangles.
Most recently,
K\"uhn, Sauermann, Steiner and Wigderson~\cite{kühn2025disproofoddhadwigerconjecture} have disproved the odd Hadwiger conjecture using our construction. Here, triangles are avoided the same way as in our application, guaranteeing a large chromatic number of the complement. Instead of avoiding large cliques in the complement, the authors show that one can avoid large odd pairings (see their paper for a definition), another pseudorandom property.

\section{The Construction of $G$}\label{sec:construction}

Before formally presenting our construction witnessing Theorem~\ref{main}, we provide some heuristic justification. 
The edge deletion method is a typical starting point in modern combinatorics when one wishes to construct a large graph~$G$ which avoids a forbidden subgraph $H$, but carries certain other properties held by the random graph (in our application,~$G$ should have no large independent sets).
Implementing the method first requires one  to calculate the edge deletion threshold $p_H$, so that (in expectation) the number of copies of $H$ in the random graph $G(n,p_H)$ is of the same order as $p_Hn^2$, the expected number of edges. If~$p\ll p_H$ and one samples~$G$ from $G(n,p)$ before deleting an edge from every copy of $H$, only~$o(pn^2)$ edges are lost.

Some further effort can show that the resulting graph will have small independence number (similar to $G(n,p)$) and many other pseudorandom properties.
For a number of problems, finding constructions which improve upon this deletion method is of great interest.
To outperform such examples, one may try to construct an $H$-free graph with edge density larger than~$p_H$, while still being careful to obtain some of the pseudorandomness that the random graph enjoys.

For triangles, the edge deletion threshold is of order $p_{K_3}=\sqrt{\frac1n}$.
Most recently, Campos, Jenssen, Michelen and Sahasrabudhe~\cite{campos2025} constructed a triangle-free graph with edge density $(1+o(1))\sqrt{\frac{2\log n}{3n}}$. To do so, they let~${m}= n / \log^2(n)$ and~$p = \sqrt{\frac{\log(n)}{6 n}}$ before sampling a graph $G_1$ from~$G({m},p)$. Noting that~$ p = o\left(\sqrt{\frac1{{m}}}\right)$, they can delete few edges to make the resulting graph (still on~${m}$ vertices) triangle-free, before blowing up each vertex by a factor $\log^2(n)$ to obtain a new graph on $n$ vertices, still triangle-free and with edge density $\sqrt{\frac{\log{n}}{6n}}$, well above the edge deletion threshold.
Unfortunately the blow-up operation also inflates independent sets. In~\cite{campos2025} it is shown that if one then runs a tailored nibble (loosely speaking, a nibble with an added regularization step), one will eventually double the density of the blow-up graph and destroy the (structured and rare) large independent sets created by the blow-up. 

The main idea in our construction is to replace the nibble in the second stage of~\cite{campos2025} by a second copy of a blow-up of the smaller random graph. 
We recover the necessary pseudorandom properties by randomly overlaying the two blow-ups; each copy of the blow-up destroys the large independent sets in the other. While there may be many resulting triangles from the overlay, they come in bunches intersecting in a single edge, and they can all be efficiently removed via the deletion of few edges. 

We learned that a similar idea of stacking triangle-free graphs to reduce the independence number was used before by Alon and R\"odl~\cite{Alon05} where they find lower bounds for multi-color Ramsey numbers $R(3,\ldots,3,k)$ by combining several copies of a Ramsey graph for $R(3,k)$. While the effect on the independence number is similar, they do not have to worry about triangles created from multiple colors.

In any graph~$G$ witnessing a lower bound to~$R(3,k)$, the neighborhood of a vertex~$v$ is necessarily an independent set.
The construction in \cite{campos2025} has independent sets asymptotically $\frac32$ times the size of the average vertex degree.
In our construction, independence number and average vertex degree asymptotically agree, and they are the same as a random graph of the same density. Any construction improving on the constant $\frac12$ would have to have lower density, and at the same time independence number smaller than the random graph of that density.

Let us now formally define our construction of the desired  $n$ vertex graph $G$. For the sake of readability, we omit floors and ceilings whenever they are inconsequential. All of our logarithms are natural with base $e$. For some (to be specified) values of $p \in (0,1)$ and $s > 1$, let ${m}=\frac{n}{s}$. 
\begin{enumerate}
    \item Sample two (independent) copies of the random graph $G({m},p)$, $G_{R}$ and $G_{B}$ (we will call these the red and the blue graph) on vertex sets $V_R=V(G_R)=\{r_1,r_2,\ldots,r_{{m}} \}$ and $V_B=V(G_B)=\{b_1,b_2,\ldots,b_{{m}} \}$, respectively.
    \item Let $V(G)=\{v_1,v_2\ldots,v_n\}$, and choose a function 
    $\pi:V(G)\to V_R\times V_B$
    \label{subgraphchoice} uniformly at random from all such injective maps. 
    See below for a discussion of this choice.
    \item Define functions $\pi_R:V(G)\to V_R$, $\pi_B:V(G)\to V_B$ by $(\pi_R(v_i),\pi_B(v_i))=\pi(v_i)$ for all $i$. For any set $X\subseteq V(G)$, define the sets $\pi(X),\pi_R(X),\pi_B(X)$ as the images of the sets under these maps. We say that $\pi_R(X)$ and $\pi_B(X)$ are the projections of $\pi(X)$.
    \item Define an edge set $\tilde{E}(G)=\tilde{E}_R(G)\cup \tilde{E}_B(G)$, where
    \begin{align*}
        \tilde{E}_R(G)&=\{v_iv_j:\pi_R(v_i)\pi_R(v_j)\in E(G_R)\},\\
        \tilde{E}_B(G)&=\{v_iv_j:\pi_B(v_i)\pi_B(v_j)\in E(G_B)\}.
    \end{align*}
    If an edge is in $\tilde{E}_R(G)\cap \tilde{E}_B(G)$, we keep both edges, so $\tilde{E}(G)$ is a multigraph.
    Edges in $\tilde{E}_R(G)$ will be colored red, edges in $\tilde{E}_B(G)$ will be colored blue. 
    \item Order all pairs $r_ir_j\in {V_R\choose 2}$ lexicographically, i.e. first by the smaller index and then by the larger index. Similarly, order all pairs $b_ib_j\in {V_B\choose 2}$.
    \item To arrive at $E(G)\subseteq \tilde{E}(G)$, we delete at least one edge from every triangle in $\tilde{E}(G)$. \label{edgedelete}
    \begin{enumerate}
        \item From every red  triangle $v_iv_jv_\ell\subseteq\tilde{E}_R(G)$, 
        with $\pi_R(v_j)\pi_R(v_\ell)$ later in the order of the pairs in $V_R$ than $\pi_R(v_i)\pi_R(v_j)$ and $\pi_R(v_i)\pi_R(v_\ell)$, remove 
        $v_jv_\ell\in \tilde{E}_R(G)$.
        \item From every blue  triangle $v_iv_jv_\ell\subseteq\tilde{E}_B(G)$ with $\pi_B(v_j)\pi_B(v_\ell)$ later in the order of the pairs in $V_B$ than $\pi_B(v_i)\pi_B(v_j)$ and $\pi_B(v_i)\pi_B(v_\ell)$, remove 
        $v_jv_\ell\in \tilde{E}_B(G)$.
        \item From every triangle $v_iv_jv_\ell$ with red edges $v_iv_j,v_iv_\ell\in\tilde{E}_R(G)$, and blue edge $v_jv_\ell\in\tilde{E}_B(G)$, remove 
        $v_jv_\ell\in\tilde{E}_B(G)$.
        \item From every triangle $v_iv_jv_\ell$ with blue edges $v_iv_j,v_iv_\ell\in\tilde{E}_B(G)$, and red edge $v_jv_\ell\in\tilde{E}_R(G)$, remove 
        $v_jv_\ell\in\tilde{E}_R(G)$.
    \end{enumerate}
    \item For our analysis, it is easier to work with this multigraph instead of the simple graph we obtain when we reduce every double edge to a single edge. As the independence number of the graph and the multigraph is the same, we will work with the multigraph.
\end{enumerate}

The resulting multigraph on $E(G)=E_R(G)\cup E_B(G)$ is triangle-free and has density $(2+o(1))p$ for an appropriate choice of $\pi$, $s$ and $p$. Note that if every vertex in $V_R\cup V_B$ appears as $\pi_R(v_i)$ or $\pi_B(v_i)$  $(1+o(1))s$ times, then each of the edges removed in step~\ref{edgedelete} lies between the end vertices of at least $(1+o(1))s$ monochromatic paths of length $2$, and therefore every edge removal destroys at least $(1+o(1))s$ triangles in this process. This factor $s$ represents exactly the gain in efficiency of this construction when compared with the typical edge deletion method. A larger value of $s$ yields a larger gain, but the resulting graph is also further from a random graph.

There are many different ways to choose the function $\pi$ in step~\ref{subgraphchoice}. Ultimately, it will be important that $\pi$ has certain expansion properties.
In this paper, we like to think of $G$ as a random subgraph of the co-normal product of $G_R$ and $G_B$ on $V_R\times V_B$ to make the notion of a projection more natural, and we choose $\pi:V(G)\to V_R\times V_B$ as an injective function uniformly at random. In~\cite{kühn2025disproofoddhadwigerconjecture}, the authors choose each $\pi(v_i)$ independently uniformly at random, allowing repetition. In~\cite{campos2025polynomialimprovementoddcyclecomplete}, the authors create the random union of balanced $s$-blow-ups of $G_R$ and $G_B$ alluded to in the heuristic description of our construction by choosing two random balanced partitions of the vertex set, which corresponds to choosing $\pi(v_i)=(\pi_R(v_i),\pi_B(v_i))$ uniformly at random allowing repetition conditioned on each $r_i$ and $b_i$ being chosen exactly $s$ times by $\pi_R$ and $\pi_B$, respectively. Each of these choices have only marginal effects on the construction and the following analysis, and are largely interchangeable in each of the three applications.

For some applications it may be useful to replace the random choice in step~\ref{subgraphchoice} by an algebraic or geometric construction, e.g. the set of incidences of a projective plane. Overall, the construction is very flexible in the choices of $\pi$, $p$ and $s$ as long as one finds a variant of step~\ref{edgedelete} to efficiently destroy all copies of a given $H$.

We now define the parameters in the construction for our application to $R(3,k)$.
For any~$n$ and~$\varepsilon$, set
\begin{align*}
	s &= \log^2(n),&
    {m} &= \frac{n}{s},&
	p &= \beta \sqrt{ \frac{\log n}{n}},&
	k &= \kappa \sqrt{n \log n},
\end{align*}    
where ultimately
$\beta = \frac{1}{2}$ and $\kappa = 1 + \varepsilon$.

In the course of analyzing the construction we will utilize the parameter hierarchy
\[
0<  \varepsilon_2 \ll \varepsilon_1 \ll \varepsilon \ll 1,
\]
and in addition, $n_0 =n_0(\varepsilon_2)\gg \frac1{\varepsilon_2}$ is taken 
to be sufficiently large so that certain inequalities hold for all $n > n_0$. We will always treat $\varepsilon,\varepsilon_1,\varepsilon_2$ as fixed, and $o(1)$ notation refers to the asymptotics as $n$ grows.

In Section~\ref{sec:prelim}, we will use~$\varepsilon_2$ in Lemma~\ref{lem:fiber_and_degree} to quantify routine concentration results in our construction, and~$\varepsilon_1$  in Lemmas~\ref{lem:large},~\ref{lem:med},~\ref{lem:small}, and~\ref{lem:huge}  
to control the quantity and location of the removed edges in $\tilde{E}(G)\setminus E(G)$.

Finally in Section~\ref{sec:independence} we will show that~$\alpha(G) < (1+\varepsilon)\sqrt{n \log n}$ with high probability. For this, we will look at the projections of a given $k$-set onto $G_R$ and $G_B$. If we disregard the deleted edges, a $k$-set in $G$ is independent if and only if both the projections in $G_R$ and $G_B$ are independent. Thus, a $k$-set with large projections is less likely to be independent than a set with small projections. On the other hand, there are many more $k$-sets with large projections than with small projections. A careful balance of these counts and probabilities together with the information gathered on the removed edges in Section~\ref{sec:prelim} then enables us to prove the desired bound.
 
Throughout, we use the standard notation $N(v)$ for the neighborhood of a vertex in a graph, where the graph to consider is often implied by the choice of the vertex $v$.

\section{Controlling the edge deletions} \label{sec:prelim}

In this section, we introduce some notation and prove several structural lemmas that enable us to analyze the $k$-sets in $G$ and help us to control the deleted edges in $\tilde{E}(G)\setminus E(G)$. For $r_i \in V_R$ and $b_i \in V_B$ let
\begin{align*}
    F(r_i) &= \pi_R^{-1}(r_i),
    &F(b_i) &= \pi_B^{-1}(b_i)
\end{align*}   
denote the respective {\em fibers} in~$V(G)$. 

The edges in $\tilde{E}_R$ incident to a vertex $v_i\in V(G)$ are independent of $\pi_B$ (and symmetrically for $\tilde{E}_B$ and $\pi_R$). In fact, if $\pi_R(v_i)=\pi_R(v_j)$, then the red neighborhoods under $\tilde{E}_R$ of $v_i$ and $v_j$ are identical. We can describe red and blue neighborhoods in $\tilde{E}$ as follows (where the $N_v^+$ will play a role when we consider edge deletions in monochromatic triangles):
\begin{align*}
    N_{r_i}&=\bigcup_{r_j\in N(r_i)}F(r_j),
    &N_{r_i}^+&=\bigcup_{j>i,r_j\in N(r_i)}F(r_j),\\
    N_{b_i}&=\bigcup_{b_j\in N(b_i)}F(b_j),
    &N_{b_i}^+&=\bigcup_{j>i,b_j\in N(b_i)}F(b_j).
\end{align*}

The following routine lemma shows that the degrees and neighborhoods in the model behave as expected with high probability. 
\begin{lemma}\label{lem:fiber_and_degree}
 Let $C=100$ 
 and~$\mathcal{D}$ be the event that the following hold for all distinct~$v,w \in (V_R \cup V_B)$.
    \begin{align} 
    \big| |F(v)| - \log^2 (n) \big| &\leq \varepsilon_2 \log^2 (n)
    \label{eqn:fiber}\\
    \big| |N(v)| - p{m} \big| & \leq \varepsilon_2 p{m}\label{deg:redblue} \\ 
	|N(v) \cap N(w)| & \leq C \log n \mbox{ ~if $v, w \in V_R$ or $v, w \in V_B$} \label{deg:codeg}\\   
    \big| |N_v|-pn \big| &\leq \varepsilon_2 pn \label{deg:eq1}\\
	|N_v \cap N_w| &\leq C \log^3 (n)  \label{deg:eq2}\\
	|\pi_B(N_v) \cap \pi_B(N_w)| &\leq 2C\log^3 (n) \mbox{ ~if $v, w \in V_R$} \label{deg:projcodegblue}
\\	|\pi_R(N_v) \cap \pi_R(N_w)| &\leq 2C \log^3 (n) \mbox{ ~if $v, w \in V_B$} \label{deg:projcodegred}
    \end{align}
    Then $\p(\mathcal{D})=1-o(1)$.
\end{lemma}
\begin{proof}[Proof sketch]
We only sketch the proof here, the details are left to the reader.
    We show that each statement fails with probability $o(1)$. Statements \eqref{deg:redblue} and \eqref{deg:codeg} are standard results about degrees and co-degrees in the random graphs $G_R$ and $G_B$. 
    Statements \eqref{eqn:fiber}, \eqref{deg:eq1} and \eqref{deg:eq2} follow from Chernoff bounds (Theorem~\ref{CH}) for $\pi(V(G))$ intersecting the various subsets of $V_R\times V_B$ determined by $G_R$ and $G_B$, taking the union over all choices of $v$ and $w$. Note that the choice of $\pi$ being injective only strengthens the concentration around the various means. Finally, for \eqref{deg:projcodegblue} and \eqref{deg:projcodegred}, we first use \eqref{deg:eq2} for a possible intersection, and then we use Chernoff bounds for the hypergeometric distribution for intersections of two random subsets of size  $(1+\varepsilon_2)pn$ of a set of size ${m}$.
\end{proof}

One main difficulty in the proof of Theorem~\ref{main} is to analyze the edges in $\tilde{E}(G)\setminus E(G)$ removed to make the graph triangle-free. Every such edge is between the two end-vertices of a monochromatic path of length $2$, so we study such vertex pairs in $N_v$ for $v\in V_R\cup V_B$.
Given $I \subseteq V(G)$ with $|I| = k$ and a vertex $v\in V_R\cup V_B$, we let
\[
X_{v}(I) = I \cap N_v,\hspace{1cm} X_{v}^+(I) = I \cap N_v^+.
\]   
For a set $I$, we call the sets of pairs 
\[
C(I)=\bigcup_{v \in V_R \cup V_B} \binom{X_v(I)}{2} ,\hspace{1cm} C^+(I)=\bigcup_{v \in V_R \cup V_B} \binom{X_v^+(I)}{2}
\]
`closed' and `closed$^+$' pairs\footnote{this nomenclature is inherited from the triangle-free process, where pairs are closed when we can not add an edge without creating a triangle}, while the pairs in $O(I)={\pi(I)\choose 2}\setminus C(I)$ and $O^+(I)={\pi(I)\choose 2}\setminus C^+(I)$ are `open' and `open$^+$', respectively.
Edges in $\tilde{E}(G)\setminus E(G)$ deleted due to monochromatic triangles are in $C^+(I)\subseteq C(I)$, whereas edges deleted due to $2$-colored triangles may be anywhere in $C(I)$.

Thus, we want to understand $C(I)$ and $C^+(I)$ for various $I$, which we do through a series of lemmas.
Due to the greater symmetry it is easier to consider only the larger sets $C(I)$, though, and all lemmas in this section contend with $C(I)$.

For a set $I$, we will partition the vertices of~$V_R \cup V_B$ by the size of $|X_v(I)|$.  Consider the cutoffs \begin{align*} 
t_1 = \frac{\sqrt{n \log n}}{\log \log n}, && t_2 = n^{1/4 + \varepsilon}, && t_3 = n^{2\varepsilon},
\end{align*} 
and partition~$V_R \cup V_B$ by size (huge, large, medium, small) as
\begin{align*} 
H_I &= \left\{v \in V_R \cup V_B: t_1 < |X_v(I)| \le k ~\right\},\\
L_I &= \left\{v \in V_R \cup V_B: t_2 < |X_v(I)| \leq t_1 \right\},\\
M_I &= \left\{v \in V_R \cup V_B: t_3 < |X_v(I)| \le t_2 \right\},\\
S_I &= \left\{v \in V_R \cup V_B: 0~ \le |X_v(I)| \le t_3 \right\}.
\end{align*}
The next four lemmas address the size of~$\binom{X_v(I)}{2}$ over each of the four parts.
For each of~$L_I,M_I$, and~$S_I$ the total contribution of closed pairs will be~$o(k^2)$, and we address these first. 

\begin{lemma} \label{lem:large} 
Conditioned on the event $\mathcal{D}$ defined in Lemma \ref{lem:fiber_and_degree}, for all $I \subseteq V(G)$ with $|I| = k$,
\[
\left| \bigcup_{v \in L_I} {X_{v}(I)\choose 2} \right|  \leq \varepsilon_1 k^2.  
\]
assuming $n>n_0$. 
\end{lemma} 
\begin{proof}
	Suppose $\mathcal{D}$ holds and fix an $I \subseteq V(G)$ with $|I| = k$. Consider any subset $L\subseteq L_I$. By inclusion-exclusion and $\eqref{deg:eq2}$ of $\mathcal{D},$
	\[
	k = |I| \geq \sum_{v \in L} |X_v(I)| - \sum_{\{v,w\} \in {L\choose 2}} |X_v(I) \cap X_{w}(I)| \geq |L|n^{1/4+\varepsilon} - \binom{|L|}{2}C\log^3(n). 
	\]
	This would yield a contradiction if $|L|=\lfloor n^{\frac{1-\varepsilon}{4}}\rfloor$.  From here, we observe that ~$|L_I| < n^{\frac{1-\varepsilon}{4}}$, and in turn that~$\sum_{v \in L_I} |X_v(I)| \leq (1+o(1))k.$  Thus,
    \[
	\sum_{v \in L_I} \binom{ |X_v(I)|}{2} \leq \sum_{v \in L_I} \frac{|X_v(I)| t_1}{2} = \frac{t_1}{2}\sum_{v \in L_I} |X_v(I)| = o(k^2),
	\] 
and the result follows for $n>n_0$. 
\end{proof}

Let us now look at the sets $X_v(I)$ of the next smaller size -- those with $v \in M_I$. 
The following is a weaker version of a lemma in~\cite{campos2025}, showing that the contribution of closed pairs from~$M_I$ is~$o(k^2)$, with high probability.

\begin{lemma} \label{lem:med} 
	Let $\mathcal{M}$ be the event that for all $I \subseteq V(G)$ with $|I|=k$,
	\begin{equation}
	\left| \bigcup_{v \in M_I} {X_{v}(I)\choose 2} \right|  \leq \varepsilon_1 k^2.  \label{eq:medcount}
	\end{equation} 
	Then $\p(\mathcal{M}) = 1-o(1).$
\end{lemma} 

\begin{proof} 

We will show that in order for \eqref{eq:medcount} to fail, there would be an unusually dense subgraph of moderate size in $G$, which is very unlikely to happen.

Fix an $I \subseteq V(G)$ with $|I| = k$. Let \[
Y_I =\sum_{v \in M_I}|X_v(I)|.\]
We prove that with high probability for all $I$, $Y_I \leq kn^{1/4-\varepsilon}.$  From this the lemma immediately follows, as 
\[
\sum_{v \in M_I} \binom {|X_v(I)|}{2} \leq  \sum_{v \in M_I} |X_v(I)|t_2 = o(k^2).
\]

Suppose that $Y_I > kn^{1/4 - \varepsilon}.$  Let $B \subseteq M_I$ be minimal so that $Y = \sum_{v \in B} |X_v(I)| > kn^{1/4 - \varepsilon}$.  Note that $Y < 2kn^{1/4 - \varepsilon}$ (since each $|X_v(I)| \leq n^{1/4 + \varepsilon}$) and that $|B| \leq n^{-2\varepsilon}Y \leq  2kn^{1/4 - 3\varepsilon}.$  Let $B_R=B\cap V_R$ and $B_B=B\cap V_B$. Then $Y$ is the number of edges in the union of two bipartite graphs $\Gamma_R$ and $\Gamma_B$ on $V(\Gamma_R)=(B_R,\pi(I))$, $V(\Gamma_B)=(B_B,\pi(I))$ where 
\[
E(\Gamma_R)=\{v(w,w'):vw \in E(G_R)\},~E(\Gamma_B)=\{v(w,w'):vw' \in E(G_B)\}.
\]
Assuming that $\mathcal{D}$ of Lemma \ref{lem:fiber_and_degree} occurs, every edge $vw\in E(G_R)$ is causing at most 
$|F(w)|+|F(v)|<3\log^2(n)$ edges $v(w,w'),w(v,v') \in E(\Gamma_R)$, and similarly for every edge $v'w'\in E(G_B)$ causing edges in $E(\Gamma_B)$.  That is, the bipartite subgraphs of $G_R$ on $(B_R, \pi_R(I))$ and of $G_B$ on $(B_B, \pi_R(I))$ 
must together contain at least $\frac{2kn^{1/4 - \varepsilon}}{6\log^2(n)}$ edges in $G_R$ and $G_B$. 

Now, consider a fixed $\tilde{B} =\tilde{B}_R\cup \tilde{B}_B\subseteq (V_R \cup V_B)$ of size $2kn^{1/4-3\varepsilon}$ and $I \subseteq V(G)$ of size $k$.  Suppose $Z \sim Bin(|\tilde{B}|k,p)$.  The probability that there are at least $\frac{kn^{1/4 - \varepsilon}}{3\log^2(n)}$ edges in the union of the two bipartite subgraphs on $(\tilde{B}_R,\pi_R(I))$ in $G_R$ and on $(\tilde{B}_B,\pi_B(I))$ in $G_B$ is at most (for $n>n_0$)
\begin{align*}
\p\left(Z \geq \frac{kn^{1/4 - \varepsilon}}{3\log^2(n)}   \right) &\leq \p\left( Z \geq \e(Z) + \frac{kn^{1/4 - \varepsilon}}{6\log^2(n)} \right) \\
& \leq \exp\left( - \frac{kn^{1/4 - \varepsilon}}{12\log^2(n)}  \right)\\
& \leq \exp\left(-n^{3/4-2\varepsilon} \right),
\end{align*} 
where here we used that $\frac{kn^{1/4 - \varepsilon}}{3\log^2(n)} \geq n^\varepsilon\e(Z)\ge  3\e(Z)$, along with Corollary \ref{CHc}.  Now, observing that
\[
\binom{n}{k}\binom{2{m}}{2kn^{1/4 - 3\varepsilon}} \leq \exp\big( k \log n + 2kn^{1/4-3\varepsilon} \log n \big)= \exp\big(o(n^{3/4-2\varepsilon})\big),
\]
and that $B\subseteq \tilde{B}$ for some $\tilde B$, a union bound over all possible choices of $\tilde{B}$ and $I$ shows that with high probability none of the potential choices of $(B,\pi(I))$ yield graphs with as many edges as required above. 

Letting $\mathcal{O}$ denote the existence of such an overfull $(B,\pi(I))$, we conclude that the probability that the high probability conclusion of the lemma is violated is at most $\p(\mathcal{O}^c \cup \mathcal{D}^{c}) = o(1).$
\end{proof}

Next, we consider the sets $X_v(I)$ with $v \in S_{I}$.   
\begin{lemma} \label{lem:small} 	
	Let $\mathcal{S}$ be the event that for all $I \subseteq V(G)$ with $|I|=k$,
	\[
	\left| \bigcup_{v \in S_I} {X_{v}(I)\choose 2} \right|  \leq \varepsilon_1 k^2.  
	\]
	Then $\p(\mathcal{S}) = 1-o(1).$ 
\end{lemma} 
\begin{proof} 
Fixing a set $I \subseteq V(G)$ with $|I|=k$ we consider
\[
	\left| \bigcup_{v \in S_I} {X_{v}(I)\choose 2} \right|  \leq \left| \bigcup_{v \in S_I \cap (\pi_R(I) \cup \pi_B(I))} \binom{X_v(I)}{2} \right| + \left| \bigcup_{v \in S_I \setminus (\pi_R(I) \cup \pi_B(I))} \binom{X_{v}(I)}{2} \right|.  
\]
Since~$|\pi_R(I) \cup \pi_B(I)| \leq 2|I| =2k$, the first term satisfies
\begin{align}\label{internalS_I}
    \left| \bigcup_{v \in S_I \cap (\pi_R(I) \cup \pi_B(I))} \binom{X_v(I)}{2} \right| \leq \binom{n^{2\varepsilon}}{2} (2k) \leq n^{4\varepsilon} k \leq \frac{\varepsilon_1}{2} k^2.  
\end{align}
The second term requires some probabilistic analysis. 
Set
\[
  Z \;:=\; \left| \bigcup_{v \in S_I \setminus (\pi_R(I) \cup \pi_B(I))} \binom{X_{v}(I)}{2} \right|.
\]
For a pair $e\in {F(w)\cap \pi(I)\choose 2}$ for some $w\in V_R\cup V_B$, $e\in {X_v(I)\choose 2}$ for any $v$ with $vw\in E(G_R)\cup E(G_B)$. Assuming $\mathcal{D}$, 
we have $|F(w)\cap \pi(I)|<2\log^2(n)$ for all $w$. As the sets ${F(w)\choose 2}$ are disjoint, and every vertex in $\pi(I)$ is in exactly two fibers,
there are fewer than $\frac{2k}{2\log^2{n}}2\log^4(n)=2k\log^2(n)$ such pairs.

For a fixed pair $e\in \binom{\pi(I)}{2}\setminus\bigcup_w {F(w)\choose 2}$, the probability that $e$ lies in the union
$\bigcup_{v} \binom{X_v(I)}{2}$ (regardless of whether or not $v\in S_I$) is at most
\[
  2\bigl(1-(1-p^2)^{m}\bigr)\;\le\; 2p^2{m} \;=\; \frac{2\beta^2}{\log n}.
\]
Whence with $\beta=\tfrac12$,
\[
  \E (Z) \leq \frac{2\beta^2}{\log n}\binom{k}{2}+2k\log^2(n)=\frac{1}{2\log n}\binom{k}{2}+2k\log^2(n).
\]

To apply McDiarmid’s inequality we will argue that~$Z$ is controlled by a family of independent random variables
\[
  Y_v := 
  \begin{cases}
      N(v) \cap (\pi_R(I) \cup \pi_B(I)),&\text{ if }v\in S_I,\\
      \emptyset,&\text{ if }v\notin S_I.
  \end{cases}
\]
indexed by~${v\in (V_R\setminus \pi_R(I))\cup (V_B\setminus \pi_B(I))}$.
For fixed $\pi$, these variables depend only on $N(v)$, which is determined by the choices of $G_R$ and $G_B$. They are mutually independent because we are only considering those where $v \not \in \pi_R(I)\cup\pi_B(I)$. As the union over such~$v$, $Z$ is a deterministic function of the~$Y_v$ and changing the value of a single such~$Y_{v}$ only alters the set $\binom{X_v(I)}{2}$ therefore changing the value of $Z$
by at most
\[
  \binom{n^{2\varepsilon}}{2} \;\le\; \frac{n^{4\varepsilon}}{2}.
\]
As there are $|S_I \setminus \pi_R(I) \cup \pi_B(I)| \leq 2{m}$ random variables in our family,   McDiarmid's Inequality (Theorem~\ref{McD}) yields, for any $t>0$,
\[
  \P\left(Z \ge \E (Z) + t\right) \leq
  \exp\left(-\,\frac{t^2}{{m}\,n^{8\varepsilon}}\right).
\]
Applying this with~$t = \frac{1}{\sqrt{\log n}}\binom{k}{2}$ and recalling~$k=\kappa\sqrt{n\log n}$, we obtain
\begin{align*}
  \P\left( Z \ge \frac{2}{\sqrt{\log n}}\binom{k}{2} \right)
  &\le \P\left( Z \ge \E(Z) + \frac{1}{\sqrt{\log n}}\binom{k}{2} \right) \\
  &\le \exp\left(-\,\frac{\bigl(\frac{1}{\sqrt{\log n}}\binom{k}{2}\bigr)^2}{{m}\,n^{8\varepsilon}}\right) \\
  &= \exp\left(-\Omega\bigl(n^{\,1-8\varepsilon}\log^{3} (n)\bigr)\right) \\
  &= o\left(\binom{n}{k}^{-1}\right).
\end{align*}
Therefore, with probability at least $1-o\left(\binom{n}{k}^{-1}\right)$,
\[
  Z \leq \frac{2}{\sqrt{\log n}}\binom{k}{2}
  \leq \frac{\varepsilon_1}{2}k^2
\]
for $n>n_0$. 
Combining this with~\eqref{internalS_I},
we find that, for a fixed $I$, the whole quantity
\[
  \left| \bigcup_{v \in S_I} {X_{v}(I)\choose 2} \right|\le \varepsilon_1 k^2
\]
with probability $1-o\left(\binom{n}{k}^{-1}\right)$.
A union bound over the $\binom{n}{k}$ choices of $I$ shows that the desired bound holds simultaneously for all~$I$.
\end{proof} 

Finally, we turn to the largest of the $X_{v}(I)$ -- those with $v \in H_I$.  Here is where the only significant contributions to the closed pairs arise. Ultimately, the independence of $I$ is determined by the behavior of $G_R$ and $G_B$ in $\pi_R(I)$ and $\pi_B(I)$, so we focus on these projected sets here.

\begin{lemma} 
If $\mathcal{D}$ of Lemma \ref{lem:fiber_and_degree} occurs, and $n>n_0$, then all $I \subseteq V(G)$ with $|I| = k$ satisfy
\begin{align} 
\sum_{v \in H_I \cap V_R} \binom{|\pi_R(X_v(I))|}{2} &\leq \varepsilon_1 k^2\label{dim8},\\
\sum_{v \in H_I \cap V_B} \binom{|\pi_B(X_v(I))|}{2} &\leq \varepsilon_1 k^2\label{dim9},\\
\sum_{v \in H_I \cap V_R} \binom{|\pi_B(X_v(I))|}{2} &\leq (1+\varepsilon_1) \min \left\{ \binom{k-|\pi_R(I)|}{2} , \binom{pn}{2}+\binom{k-|\pi_R(I)|-pn}{2}\right\}\label{dim6},\\
\sum_{v \in H_I \cap V_B} \binom{|\pi_R(X_{v}(I))|}{2} &\leq (1+\varepsilon_1)\min\left\{ \binom{k-|\pi_B(I)|}{2} , \binom{pn}{2}+\binom{k-|\pi_B(I)|-pn}{2}\right\}\label{dim7}. 
\end{align} 
\label{lem:huge} 
\end{lemma}   

\begin{proof}
	Assuming $\mathcal{D}$, an analogous proof to that of Lemma \ref{lem:large} implies that \begin{equation}
	\sum_{v \in H_I} |X_v(I)| \leq (1+o(1)) \left| \bigcup_{v \in H_I} X_v(I)\right|.
	\label{eqn:sumxv}\end{equation} 
    Indeed, for any $H \subseteq H_I$ 
    \[
   k \geq \left| \bigcup_{v \in H_I} X_v(I)\right| \geq \sum_{v \in H} |X_v(I)| - \sum_{\{u, v\} \in {H\choose 2}} |X_v(I) \cap X_u(I)| \geq |H| \frac{\sqrt{n \log n}}{\log \log n} - {|H| \choose 2} C\log^3 (n).  
    \]
This is a contradiction for any $|H| = (1+o(1)) k \frac{\log \log n}{\sqrt{ n \log n}}$, implying that 
\begin{align}\label{smallH_I}
    |H_I| \leq (1+o(1)) k \frac{\log \log n}{\sqrt{ n \log n}} < 2\log \log n
\end{align} 
and hence~\eqref{eqn:sumxv}.      
  
  From this and the fact (from $\mathcal{D}$) that all vertices in $V_R$ and $V_B$ have degree at most $(1+\varepsilon_2)p{m} \leq  (\beta+\varepsilon_2)\sqrt{\frac{n}{\log^3 (n)}}$ we have   
  	\begin{align} 
		\sum_{v \in H_{I} \cap V_R} |\pi_R(X_v(I))| \leq |H_I|(1+\varepsilon_2)p{m} = o(k) && \text{and} && 	\sum_{v \in H_{I} \cap V_B} |\pi_B(X_v(I))| = o(k). \label{eq:projbound}
	\end{align} 
This already implies~\eqref{dim8} and~\eqref{dim9}.
    Now, if $u, v \in H_I \cap V_R$, part \eqref{deg:projcodegblue} of $\mathcal{D}$ shows that $\pi_B(X_u(I))$ and $\pi_B(X_v(I))$ have small intersection, from which one gets with 
    ~\eqref{smallH_I} that
    \begin{align*}
    \sum_{v \in H_I\cap V_R} |\pi_{B}(X_v(I))| = \left| \bigcup_{v \in H_I\cap V_R} \pi_B(X_v(I))\right|+o(k).
    \end{align*} 

    Order the vertices of $H_I \cap V_R$ as $v_1, \dots, v_t$, and set $x_{v_i} = |\pi_B(X_{v_i}(I))|$. We may assume that $x_{v_1}$ is the maximum of the $x_{v_i}$. Note that $x_{v_1} \leq |N_3(v_1)| \leq (1+\varepsilon_2)pn$.  
    By the convexity of the binomial coefficients, we have for 
    $n>n_0$
    \begin{align*}
        \sum \binom{x_{v_i}}{2}&\le {\sum x_{v_i}\choose 2}\\
        &\le {(1+\varepsilon_2)\left| \bigcup \pi_B(X_{v_i}(I))\right|\choose 2}\\
        &\le (1+3\varepsilon_2){\left| \bigcup \pi_B(X_{v_i}(I))\right|\choose 2},
    \end{align*}
    and for $\sum x_{v_i}>(1+\varepsilon_2)pn$,
\begin{align*}
        \sum \binom{x_{v_i}}{2}&\le  {x_{v_1}\choose 2}+\sum_{i\ge 2} \binom{x_{v_i}}{2}\\
        &\le {(1+\varepsilon_2)pn\choose 2}+{\sum x_{v_i}-(1+\varepsilon_2)pn\choose 2}\\
        &\le {(1+\varepsilon_2)pn\choose 2}+{(1+\varepsilon_2)\left(\left| \bigcup \pi_B(X_{v_i}(I))\right|-pn\right)\choose 2}\\
        &\le (1+3\varepsilon_2)\left( {pn\choose 2}+{\left| \bigcup \pi_B(X_{v_i}(I))\right|-pn\choose 2}\right).
    \end{align*}
    \noindent    Finally, to complete the proof of \eqref{dim6}, remember that $\varepsilon_2$ is sufficiently small compared to $\varepsilon_1$, and  bound 
    \[
    \left|\bigcup \pi_{B}(X_{v_i}(I))\right| \leq \left|\bigcup X_{v_i}(I) \right| \leq k - |\pi_R(I)| + o(k).  \]
This last inequality follows from \eqref{eq:projbound}, 
    observing that
    \begin{align*}
    |\pi_R(I)| &\leq \left|\pi_R\left( \bigcup X_{v_i}(I) \right) \right|+\left|\pi_R\left(\pi(I) \setminus \left(\bigcup X_{v_i}(I)\right) \right) \right| \\ 
    &\leq \left|\pi_R\left( \bigcup X_{v_i}(I) \right) \right|+\left|\pi(I) \setminus \left(\bigcup X_{v_i}(I)\right)  \right| \\ 
    &= o(k) +k-\left|\bigcup X_{v_i}(I)\right|.
    \end{align*}
The proof of \eqref{dim7} works analogously, considering $v \in H_I \cap V_B$.  
\end{proof}

\section{Independent sets in $G$ and the proof of Theorem \ref{main}}\label{sec:independence}

We now turn to the crucial part of Theorem \ref{main}, showing that vertex sets of size $k$ are not independent.  The only time that it is important to remove edges from monochromatic triangles by considering only $C^+(I)$ is in one place in the proof of Lemma~\ref{lem:RISI} below. For all previous arguments, we could have omitted all edges in monochromatic triangles since we bound $C(I)$ instead.

For what follows, fix a $k$-set $I \subseteq V(G)$.   
We will begin by revealing the edges in $G_R$ and $G_B$ incident to vertices not in $\pi_R(I) \cup \pi_B(I)$ and a small selection of edges within $\pi_R(I) \cup \pi_B(I)$. We ensure that most open pairs in $\pi_R(I) \cup \pi_B(I)$ remain to be considered while most closed pairs are already revealed. As all open pairs must be non-edges, this allows us to bound the probability that 
the set is independent at the end of the process.  
Consider the following procedure.  
\begin{enumerate} 
\item First, let $F_0= (V_R\cup V_B)\setminus (\pi_R(I) \cup \pi_B(I))$. For all vertices $v \in F_0$, reveal all neighbors and non-neighbors in $\pi_R(I)$ and $\pi_B(I)$ in $G_R$ and $G_B$ respectively.  These, in turn, reveal edges in $\tilde{E}(G)$.  

\item Let $F_1=\{v\in \pi_R(I)\cup\pi_B(I):|F(v) \cap I|>\log n\}$. For all vertices in $F_1$, reveal all neighbors and non-neighbors in $\pi_R(I)$ and $\pi_B(I)$ in $G_R$ and $G_B$ respectively.  These, in turn, reveal edges in $\tilde{E}(G)$.

\item Let $F_2=\{v\in \pi_R(I)\cup\pi_B(I):|N(v)\cap F_1|>\frac{t_2}{\log n}\}$. Note that 
\[(\pi_R(I)\cup\pi_B(I)) \cap H_I\subseteq F_2\subseteq L_I\cup H_I.\]   Indeed, for $v \in \pi_R(I) \cup \pi_B(I)$ with $v \not\in F_2$ then \[
|X_v(I)| < (1+\varepsilon_2)(\log^2 n)\frac{t_2}{\log n} + (\log n) (1 + \varepsilon_2)p{m} \ll t_1,
\]
so $v \not\in H_I$, proving the first inclusion.  The second inclusion follows as $|N(v) \cap F_1| > \frac{t_2}{\log n}$ along with the definition of $F_1$ implies that $|X_v(I)| > t_2$.

For all vertices in $F_2$, reveal all neighbors and non-neighbors in $\pi_R(I)$ and $\pi_B(I)$ in $G_R$ and $G_B$ respectively.  These, in turn, reveal edges in $\tilde{E}(G)$.

\item Reveal all pairs in $X_v(I)$ for $v\in F=F_0\cup F_1\cup F_2$.
\end{enumerate} 
To reduce notation, let $\mathcal R=\mathcal D\cap\mathcal S\cap \mathcal M$ in the following.
Let 
\begin{multline*}
f(\ell_R, \ell_B) =  {\ell_R \choose 2} + {\ell_B \choose 2} -  \min\left\{ \binom{k-\ell_R}{2} , \binom{pn}{2}+\binom{k-\ell_R-pn}{2}\right\}\\
- \min\left\{ \binom{k-\ell_B}{2} , \binom{pn}{2}+\binom{k-\ell_B-pn}{2}\right\}.
\end{multline*}
For integers $\ell_R,\ell_B\in [k]$ with $\ell_R\cdot\ell_B\ge k$, let $\mathcal{S}_{I,\ell_R,\ell_B}$ denote the event that $|\pi_R(I)| = \ell_R$ and $|\pi_B(I)| = \ell_B.$  By Lemmas~\ref{lem:large}-\ref{lem:huge}, if $\mathcal{R} \cap \mathcal{S}_{I,\ell_R,\ell_B}$ holds, then $\pi_R(I)\cup \pi_B(I)$ contains at least $f(\ell_R,\ell_B)-9\varepsilon_1k^2$ 
open pairs. As $|F_1|\le 2\frac{k}{\log n}$ and $|F_2|\le |L_I\cup H_I|< n^{\frac14} $ by the proof of Lemma~\ref{lem:large}, the procedure reveals at most $k|F_1\cup F_2|<\varepsilon_1k^2$ further pairs in $O(I)$. Thus at this point, there are at least $f(\ell_R,\ell_B)-10\varepsilon_1k^2$ open pairs unrevealed.

We observe
\begin{lemma}\label{lem:RISI}
Suppose $I \subseteq V(G)$ 
with $|I| = k$.  Let $\mathcal{B}_I$ be the event that $I$ is independent in $G$.  Then, assuming $n>n_0$, 
\[
\p(\mathcal{B}_I \cap \mathcal R\mid   \mathcal{S}_{I, \ell_R, \ell_B}) \leq (1-p)^{f(\ell_R, \ell_B) - \varepsilon^3 k^2} \leq \exp(-p(f(\ell_R, \ell_B) - \varepsilon^3k^2)).  
\]
\end{lemma} 

\begin{proof}
Let $E_{I,R}$ and $E_{I,B}$ be the sets of pairs not revealed so far in $G_R$ and $G_B$, and $E_I=E_{I,R}\cup E_{I,B}$.
The key observation regards how $I$ can be independent at the end.  Recall that 
if $I$ is independent, all edges in $E_{I}$ must be in $\pi_R(C(I))\cup \pi_B(C(I))$. More precisely, all edges in $E_{I,R}\cap E(G_R)$ must be in some ${N^+(r_i)\choose 2}$ or in some $\pi_R\left({N_{b_i}\choose 2}\right)$, and all edges in $E_{I,B}\cap E(G_B)$ must be in some ${N^+(b_i)\choose 2}$ or in some $\pi_B\left({N_{r_i}\choose 2}\right)$.

Let \begin{align*}
    T_R&= E_{I,R}\setminus \bigcup_{i=1}^{m} {N^+(r_i)\choose 2},
    &T_B&= E_{I,B}\setminus \bigcup_{i=1}^{m} {N^+(b_i)\choose 2}.
\end{align*}
Let $U_R=T_R\cap E(G_R)$, $U_B=T_B\cap E(G_R)$.
These are the edges in $E_I$ which are not closed$^+$ by vertices of the same color, and thus must be closed by vertices of the opposite color. For $0\le u_R,u_B\le k^2$, denote the event that $|U_R|=u_R$ and $|U_B|=u_B$ by $\mathcal U_{I,u_R,u_B}$. We have 
\begin{align}
    \p(\mathcal{B}_I \cap \mathcal  R\mid  \mathcal{S}_{I, \ell_R, \ell_B}) =
\sum_{u_R}\sum_{u_B}\p(\mathcal{B}_I \cap \mathcal R\cap \mathcal U_{I,u_R,u_B}\mid  \mathcal{S}_{I, \ell_R, \ell_B}) .\label{unicorn0}
\end{align}
Assume by symmetry that $u_R\ge u_B$. Then
\begin{align}
    \p(\mathcal{B}_I\cap \mathcal R\cap \mathcal U_{I,u_R,u_B} \mid  \mathcal{S}_{I, \ell_R, \ell_B})
    &\le {3\varepsilon_1 k^2 \choose u_R} p^{u_R} {|T_B| \choose u_B} p^{u_B} (1-p)^{|T_R|+|T_B|-u_R-u_B}\label{unicorn1}\\
    &\le {3\varepsilon_1 k^2 \choose u_R} p^{u_R} {k^2 \choose u_B} p^{u_B} (1-p)^{f(\ell_R, \ell_B)-15\varepsilon_1k^{2}}\label{unicorn2}\\
	&\leq  \left( \frac{ 3e\varepsilon_1 p k^2}{u_R}\right)^{u_R}   \left( \frac{epk^2}{u_B} \right)^{u_B}  (1-p)^{f(\ell_R, \ell_B)-15\varepsilon_1k^{2}}\nonumber\\
	&\leq \left( \frac{ 3e\sqrt{\varepsilon_1} p k^2}{u_R}\right)^{u_R}   \left( \frac{e\sqrt{\varepsilon_1}pk^2}{u_B} \right)^{u_B}  (1-p)^{f(\ell_R, \ell_B)-15\varepsilon_1k^{2}}\label{unicorn4}\\
	&\leq \exp\Big( 4\sqrt{\varepsilon_1} pk^2 - p(f(\ell_R,\ell_B)-10\varepsilon_1 k^2)\Big)\label{unicorn5}.
\end{align}
To see this, reveal first the pairs in $E_{I,B}$ in lexicographic order, and then the pairs in $E_{I,R}$. It is critically important that whenever we reach a pair in $E_{I}$ in this order, it is already determined if the pair is in $T_R\cup T_B$. 
For~\eqref{unicorn1}, note that there are ${|T_B|\choose u_B}$ choices for $U_B$, and for each of these choices each of these pairs is an edge of $G_B$ with probability $p$. For all other pairs in $T_B$, the pair must be a non-edge.
Then, the edges in $U_R$ have to be chosen from the pairs in 
\[
\bigcup_{b_i\in \pi_B(I)\setminus F}\pi_R\left({X_{b_i}(I)\choose 2}\right),
\]
a set that is completely determined once we have revealed all of $G_B$. As all vertices in $\pi_B(I)\setminus F$ are in $L_I\cup M_I\cup S_I$, Lemmas~\ref{lem:large}-\ref{lem:small} bound the number of such pairs to $3\varepsilon_1 k^2$ if $\mathcal R$ holds. Finally, all other pairs in $T_R\cup T_B$ must be non-edges.

For~\eqref{unicorn2}, we observe that trivially $|T_B|\le k^2$, and that all open pairs in $E_I$ are non-edges in $T_I$. Inequality~\eqref{unicorn4} is due to $u_R\ge u_B$, and~\eqref{unicorn5} follows from the fact that $(C/x)^x$ is maximized for $x=C/e$. 
Therefore by~\eqref{unicorn0},
\begin{align*}
\p(\mathcal{B}_I \cap \mathcal R\mid  \mathcal{S}_{I, \ell_R, \ell_B}) 
&\leq 
k^4 \exp\Big( 8\sqrt{\varepsilon_1} pk^2 - p(f(\ell_r,\ell_B)-10\varepsilon_1 k^2)\Big)\\
&=\exp\Big(  - p(f(\ell_r,\ell_B)-8\sqrt{\varepsilon_1} k^2 -10\varepsilon_1 k^2)+4\log k\Big),
\end{align*}
which, for sufficiently small $\varepsilon_1$ and $k>\sqrt{n_0}$ implies the result.
\end{proof}

We next resolve the dependence on $\mathcal{S}_{I, \ell_R, \ell_B}$.
\begin{lemma}\label{lem:RI}
    Suppose $I \subseteq V(G)$ 
    with $|I| = k$.  
    Then
\[
\p(\mathcal{B}_I \cap \mathcal R ) =o(1){n\choose k}^{-1}.  
\]
\end{lemma}
\begin{proof}
Note that since $\pi$ is chosen uniformly at random,
$\pi(I)$ is a uniformly chosen $k$-subset of $V_R \times V_B$.  
    For any $\ell_R,\ell_B\in [k]$ with $\ell_R\cdot\ell_B\ge k$, and writing $\ell_R=x_Rk$, $\ell_B=x_Bk$, we have
    \begin{align*}
        \p(\mathcal{S}_{I, \ell_R, \ell_B}) &=\p(|\pi_R(I)|=\ell_R\land |\pi_B(I)|=\ell_B)\\
        &\le \p(|\pi_R(I)|\le \ell_R\land |\pi_B(I)|\le \ell_B)\\
        &\le {{m} \choose \ell_R} {{m} \choose \ell_B} \frac{{\ell_R\ell_B \choose k}}{{{m}^2 \choose k}}\\
        &\le e^{k} {{m} \choose \ell_R} {{m} \choose \ell_B} \left( \frac{\ell_R}{{m}} \right)^{k} \left( \frac{\ell_B}{{m}} \right)^{k}\\
        &\le {m}^{(x_R+x_B-2)k}k^{(-x_R-x_B+2)k}e^{(1+x_R+x_B)k}x_R^{(1-x_R)k}x_B^{(1-x_B)k}\\
        &=n^{-\frac12(2-x_R-x_B)k+o(k)}\\
        &= \exp\left( -\frac{2-x_R-x_B}{2}(1+o(1))k \log {n}\right).
    \end{align*}
    Therefore, using that ${n\choose k}=n^{\frac12 k+o(k)}$,
    \begin{align}
        \p(\mathcal{B}_I \cap \mathcal R) {n\choose k}
        &\le k^2\max_{\ell_R,\ell_B}\p(\mathcal{B}_I \cap \mathcal R\mid  \mathcal{S}_{I, \ell_R, \ell_B}) \p(\mathcal{S}_{I, \ell_R, \ell_B}) {n\choose k} \nonumber\\
        &\le \max_{\ell_R,\ell_B}\exp\left(-\frac{1-x_R-x_B}{2}(1+o(1))k \log {n}\right)\p(\mathcal{B}_I \cap \mathcal R\mid \mathcal{S}_{I, \ell_R, \ell_B}).\label{eq:long}
    \end{align}
    We consider three cases based on $x_R+x_B$. 
    If $x_R+x_B\le 1-\frac{\varepsilon}{2}$ at the maximum, then we use the trivial bound $\p(\mathcal{B}_I \cap \mathcal R\mid  \mathcal{S}_{I, \ell_R, \ell_B})\le 1$ to show that
\begin{align*}
    \eqref{eq:long}&\le  \exp \left(-\frac{\varepsilon}{4}k\log n + o(k\log n)\right)\\
    &=o(1).
\end{align*}
If $x_R+x_B\ge 1+\frac{\varepsilon}{2}$ at the maximum,  then by Lemma \ref{lem:RISI} and the values $\beta=\frac12$ and $\kappa=1+\varepsilon$,
\begin{align*}
    \eqref{eq:long}&\leq  \exp \left( -\frac{1-x_R-x_B}{2}(1+o(1))k \log {n} -p\left( \binom{\ell_R}{2} +\binom{\ell_B}{2}- \binom{k-\ell_R}{2} -\binom{k-\ell_B}{2} -\varepsilon^3k^2\right) \right)\\
        &\leq \exp \left(\left( \frac{x_R+x_B-1}{2} - \frac{\beta (-2\kappa+2\kappa(x_R+x_B)-2\varepsilon^3\kappa)}{2}\right)k\log n +o(k\log n)\right)\\
        &=\exp \left( \left( -\frac{\varepsilon(x_R+x_B-1-\varepsilon^3-\varepsilon^4) }{2}\right)k\log n +o(k\log n)\right)\\
        &\le \exp \left( - \left( \frac{\varepsilon^2(1-2\varepsilon^2-2\varepsilon^3)}{4}\right)k\log n +o(k\log n)\right)\\
        &=o(1).
        \end{align*}
Finally, if $1-\frac{\varepsilon}{2}<x_R+x_B<1+\frac{\varepsilon}{2}$ at the maximum, then we may assume by symmetry that $x_B\le x_R$ and thus $x_B\le \frac{x_R+x_B}{2}< \frac12+\frac{\varepsilon}{4}$, which implies that 
$k-\ell_B>pn$ and 
\[
 {k-\ell_B\choose 2} > {pn\choose 2}+{k-\ell_B-pn\choose 2}.
\]
Therefore, by Lemma \ref{lem:RISI} and the values $\beta=\frac12$ and $\kappa=1+\varepsilon$,
\begin{align*}
    \eqref{eq:long}&\leq  \exp \left( -\frac{1-x_R-x_B}{2}(1+o(1))k \log {n} -p\left( \binom{\ell_R}{2} +\binom{\ell_B}{2}- \binom{k-\ell_R}{2} - {pn\choose 2}-{k-\ell_B-pn\choose 2}-\varepsilon^3k^2\right) \right)\\
    &= \exp\left( \left(-\frac{1-x_R-x_B}{2}-\frac{\beta}{2\kappa}\left( -2\kappa^2+2\kappa^2(x_R+x_B)+2\kappa\beta-2x_B\kappa\beta-2\beta^2-2\varepsilon^3\kappa^2\right)\right)k\log{n}+o(k\log {n})    \right)\\
    &= \exp\left(\frac{1}{8(1+\varepsilon)}\left( 2x_B-1-(4\varepsilon^2+2\varepsilon)(x_R+x_B-1)-2\varepsilon x_R+8\varepsilon^3(1+\varepsilon)^2\right) k\log{n}+o(k\log {n})\right)\\
    &\le \exp\left(\frac{1}{8(1+\varepsilon)}\left(2(1-\varepsilon)x_B-1-(4\varepsilon^2+2\varepsilon)(x_R+x_B-1)+9\varepsilon^3\right)k\log{n}+o(k\log {n})\right)\\
    &\le \exp\left(\frac{1}{8(1+\varepsilon)}\left(2(1-\varepsilon)\frac{2+\varepsilon}{4}-1+2\varepsilon^3+\varepsilon^2+9\varepsilon^3\right)k\log{n}+o(k\log {n})\right)\\
    &=\exp\left(\frac{\varepsilon(-1+\varepsilon+22\varepsilon^2)}{16(1+\varepsilon)}k\log{n}+o(k\log {n})\right)\\
        &=o(1).
        \end{align*}
\end{proof}

We are finally ready to prove Theorem \ref{main}.

\begin{proof}[Proof of Theorem \ref{main}]
Now we estimate the probability that $G$ has an independent set of size $k$.  This is
\begin{align*} 
\p \left( \bigcup_{I \in {V(G)\choose k}} \mathcal{B}_I \right) &=  \p\left( \bigcup_{I \in {V(G)\choose k}} \mathcal{B}_I \cap \mathcal{R} \right) +\p\left( \bigcup_{I \in {V(G)\choose k}} \mathcal{B}_I \cap \mathcal{R}^c \right)\\
&\le \p\left( \bigcup_{I \in {V(G)\choose k}} \mathcal{B}_I \cap \mathcal{R} \right)+\p(\mathcal R^c)\\
&\le \sum_{I \in {V(G)\choose k}} \p(B_I\cap \mathcal R)+o(1)\\
&=o(1)
\end{align*} 
by Lemma \ref{lem:RI}.  Thus with probability $1-o(1)$, $G$ has $\alpha(G) < k =(1+\varepsilon)\sqrt{n \log n}.$ 
\end{proof}

\section{The hypergraph Ramsey number $R \bigl( S_4^{(3)},S_k^{(3)}\bigr)$} \label{sec:hyper}
In this section, we show that our method can also be applied in other settings, and we apply it to hypergraphs.
Ramsey numbers for hypergraphs are defined analogously to Ramsey numbers for graphs. Define the $3$-uniform star hypergraph $S_k^{(3)}$ as the hypergraph on $k$ vertices containing exactly all edges incident to a given central vertex. In particular, the hypergraph on four vertices with three edges is $S_4^{(3)}=K_4^{(3)-}$, the complete hypergraph minus one edge. The link of a vertex $v$ in a $3$-uniform hypergraph is the graph induced by the edges containing $v$.

Therefore, if we want to find a $2$-coloring of the complete hypergraph containing no red $S_4^{(3)}$ and no blue $S_k^{(3)}$, we are looking for a (red) hypergraph in which no link of any vertex contains a triangle or an independent set of size $k-1$. In other words, we are looking for a hypergraph in which every vertex link is a Ramsey graph for $R(3,k-1)$. This immediately implies that $R\bigl(S_4^{(3)},S_k^{(3)}\bigr)\le R(3,k-1)+1$.

Mubayi suggested to us that with only minor modifications, our proof for the lower bound for $R(3,k)$ also gives a lower bound for $R\bigl(S_4^{(3)},S_k^{(3)}\bigr)$, improving the implied constant in a recent paper by Mubayi and Spanier~\cite{mubayi2025k4freetriplesystemslarge}, which proves a conjecture by Conlon, Fox, He, Mubayi, Suk and Verstra\"ete~\cite{Conlon23} using a nibble argument. We outline the proof below, leaving the details to the reader.

\begin{theorem}\label{thm:hyper}
    $\left(\frac12-o(1)\right)\frac{k^2}{\log k}\le R\bigl(S_4^{(3)},S_k^{(3)}\bigr)\le \left(1+o(1)\right)\frac{k^2}{\log k}$.
\end{theorem}
\begin{proof}[Proof sketch]
In a similar way as we have constructed $G$, we construct a random hypergraph $H$ as follows:

\begin{enumerate} 
\item Independent $H({m},p)$ random $3$-graphs $H_{R}$ and $H_{B}$ are generated on vertex sets $V_R = \{r_1, \dots, r_{{m}}\}$ and $V_B = \{b_1, \dots, b_{{m}}\}$, where $p = \beta \sqrt{\frac{\log n}{n}}.$  Edges in $H_R$ and $H_B$ are referred to as red and blue edges respectively.  
 \item Let $V(H)=\{v_1,\ldots,v_n\}$. Choose an injection $\pi:V(H)\rightarrowtail V_R \times V_B$ uniformly at random. A triple $v_iv_jv_\ell$ is a red edge in $\tilde{E}_R(H)$ if for $(\pi(v_i),\pi(v_j),\pi(v_\ell))=((r,b),(r',b'),(r'',b''))$, $rr'r'\in E(H_R)$ and $b,b',b''\in V_B$ are three different vertices. Blue edges $\tilde{E}_B(H)$ are defined symmetrically.
\item Create the final $S_4^{(3)}$-free (multi)graph $H$ with edge set $E(H)\subseteq \tilde{E}(H)=\tilde{E}_R(H)\cup \tilde{E}_B(H)$ as follows:  
\begin{itemize} 
	\item For any red copy of $S_4^{(3)}$, out of the three red edges in $\tilde{E}_R(H)$ considered in lexicographical order (i.e., edges are ordered first by their lowest red coordinate of their image, then by their middle red coordinate, and then by the highest red coordinate), the last edge is deleted.  
	\item The same is done with blue copies of $S_4^{(3)}$, considering the lexicographic order on the blue edges.  
        \item From any non-monochromatic $S_4^{(3)}$ with exactly two blue edges, the red edge is removed.
	\item From any non-monochromatic $S_4^{(3)}$ with exactly two red edges, the blue edge is removed.
\end{itemize}  
\end{enumerate}

The link of any vertex $v\in V(H)$ is exactly the same as the graph $G$ we constructed for $R(3,k)$. Considering pairs $(v,I)$ instead of $I$ in $V(G)$ yields the exact same computations as in Section~\ref{sec:prelim}. The extra factor of $n$ when counting these pairs can always be absorbed in the error terms.

The only addition is that some edges in the link of $v$ may be deleted due to $S_4^{(3)}$ in $\tilde{E}(H)$ not centered at $v$, and these are not captured in our analysis of the sets $X_w(I)$ inside the link of $v$. For this to happen and $\pi(v)=(v_R,v_B)$, we would need three vertices $x,y,z\in V_R$ (or $V_B$) with edges $v_Rxy,xyz\in E(H_R)$ (or $v_Bxy,xyz\in E(H_B)$) in which case the pair $yz$ would be closed in the link of $v$. For every $k$-set $\pi(I)\subseteq V_R\times V_B$, we expect at most
\[
\pi_R(I)^2{m}p^2\le \frac{k^2}{4}\frac{n}{\log^2(n)}\frac{\log n}{n}=\frac{k^2}{4\log n}
\]
such closed pairs $yz$ in $\pi(I)$ coming from $H_R$, and the same bound from $H_B$. A Chernoff bound then shows that with high probability the number of these pairs is $o(k^2)$ for all $(v,I)$. 

Before the proof analogous to Lemma~\ref{lem:RISI}, we first reveal all edges of $H_R$ and $H_B$ not containing $v_R$ and $v_B$, then we proceed as in the graph case with edges in the link not contained in $\pi_R(I)\cup \pi_B(I)$, and then iteratively reveal all edges in closed pairs in the link until the process stops.

The proof of Lemma~\ref{lem:RISI} and the remaining computations are the same, concluding the proof sketch.
\end{proof}

\section{Conclusion}

In any~$n$-vertex triangle free graph~$G$ we have that~$\alpha(G) \geq \Delta(G) \geq \frac{2e(G)}{n}$. Our construction has~$\alpha(G) = (1+o(1))\frac{2e(G)}{n} = (1+o(1))\sqrt{n \log n}$; furthermore these independent sets appear both as the neighborhoods of single vertices, as well as `mixed' throughout~$G$, resembling those in a random graph of the same density.
This is one of the main supporting heuristics for Conjecture~\ref{con1/2}, and we agree that Shearer's~\cite{Shearer83} upper bound on $R(3,k)$ can likely be improved. 
 Improving Shearer's Theorem is related to the well-studied question 
 of whether there is an efficient algorithm for finding large independent sets in the random graph $G(n,p)$. 
For most given densities, known efficient algorithms can find only independent sets of about half the size of the independence number.

Shearer's algorithmic lower bound on the independence number is
$\alpha(G)\ge (1+o(1)) n\frac{\log d}{d}$,
where $d$ is the average degree of a triangle-free graph.
This is half of our upper bound when using the average degree in our construction, matching the gap that persists for random graphs. For the best general lower bound on the independence number, one then
uses $\alpha(G)\ge d$ and optimizes the average degree.
 
This gives further support to Conjecture~\ref{con1/2}: if one believes that maximum Ramsey graphs have certain pseudorandom properties, it is unsurprising that a lower bound implied by a bound from a greedy algorithm based on the average degree would be smaller than the actual independence number by a factor of $\sqrt{\frac12}$.

\section{Acknowledgments}

An early version of the construction of $G$ was first considered by the fourth author after listening to Marcus Michelen's beautiful talk about their new lower bound at the ``Rocky Mountain Summer Workshop 2025" at Colorado State University, supported by NSF grant CCF-2309707, followed by stimulating discussions over the next days.

The authors then started seriously working on this project at the ``Graduate Workshop in Combinatorics" which was held a couple weeks later at Iowa State University in Ames, Iowa, and was supported by NSF grant DMS-2152490 and the Combinatorics Foundation. 

We would like to thank David Conlon, Dhruv Mubayi, Will Perkins and Yuval Wigderson for helpful comments.

\bibliographystyle{abbrvurl}
\bibliography{refs.bib}

@Inbook{Spencer2011,
author="Spencer, Joel",
editor="Soifer, Alexander",
title="Eighty Years of Ramsey $R(3, k){\ldots}$ and Counting!",
bookTitle="Ramsey Theory: Yesterday, Today, and Tomorrow",
year="2011",
publisher="Birkh{\"a}user Boston",
address="Boston, MA",
pages="27--39",
isbn="978-0-8176-8092-3",
doi="10.1007/978-0-8176-8092-3_2",
url="https://doi.org/10.1007/978-0-8176-8092-3_2"
}

@misc{campos2025,
      title={A new lower bound for the {R}amsey numbers ${R}(3,k)$}, 
      author={Marcelo Campos and Matthew Jenssen and Marcus Michelen and Julian Sahasrabudhe},
      year={2025},
      eprint={2505.13371},
      archivePrefix={arXiv},
      primaryClass={math.CO},
      url={https://arxiv.org/abs/2505.13371}, 
}

@article {BohmanKeevash21,
    AUTHOR = {Bohman, Tom and Keevash, Peter},
     TITLE = {Dynamic concentration of the triangle-free process},
   JOURNAL = {Random Structures Algorithms},
  FJOURNAL = {Random Structures \& Algorithms},
    VOLUME = {58},
      YEAR = {2021},
    NUMBER = {2},
     PAGES = {221--293},
      ISSN = {1042-9832,1098-2418},
   MRCLASS = {05C80 (05C35 60C05)},
  MRNUMBER = {4201797},
       DOI = {10.1002/rsa.20973},
       URL = {https://doi.org/10.1002/rsa.20973},
}

@article {Fiz20,
    AUTHOR = {Fiz Pontiveros, Gonzalo and Griffiths, Simon and Morris,
              Robert},
     TITLE = {The triangle-free process and the {R}amsey number {$R(3,k)$}},
   JOURNAL = {Mem. Amer. Math. Soc.},
  FJOURNAL = {Memoirs of the American Mathematical Society},
    VOLUME = {263},
      YEAR = {2020},
    NUMBER = {1274},
     PAGES = {v+125},
      ISSN = {0065-9266,1947-6221},
      ISBN = {978-1-4704-4071-8; 978-1-4704-5656-6},
   MRCLASS = {05C35 (05C55 05C69 05C75 05C80)},
  MRNUMBER = {4073152},
MRREVIEWER = {Jan\ Hladk\'y},
       DOI = {10.1090/memo/1274},
       URL = {https://doi.org/10.1090/memo/1274},
}

@article {Shearer83,
    AUTHOR = {Shearer, James B.},
     TITLE = {A note on the independence number of triangle-free graphs},
   JOURNAL = {Discrete Math.},
  FJOURNAL = {Discrete Mathematics},
    VOLUME = {46},
      YEAR = {1983},
    NUMBER = {1},
     PAGES = {83--87},
      ISSN = {0012-365X,1872-681X},
   MRCLASS = {05C99},
  MRNUMBER = {708165},
MRREVIEWER = {Linda\ Lesniak},
       DOI = {10.1016/0012-365X(83)90273-X},
       URL = {https://doi.org/10.1016/0012-365X(83)90273-X},
}

@article {AjtaiKSz81,
    AUTHOR = {Ajtai, Mikl\'os and Koml\'os, J\'anos and Szemer\'edi, Endre},
     TITLE = {A dense infinite {S}idon sequence},
   JOURNAL = {European J. Combin.},
  FJOURNAL = {European Journal of Combinatorics},
    VOLUME = {2},
      YEAR = {1981},
    NUMBER = {1},
     PAGES = {1--11},
      ISSN = {0195-6698,1095-9971},
   MRCLASS = {10L10 (05A15)},
  MRNUMBER = {611925},
MRREVIEWER = {Marthe\ Grandet},
       DOI = {10.1016/S0195-6698(81)80014-5},
       URL = {https://doi.org/10.1016/S0195-6698(81)80014-5},
}

@article {AjtaiKSz80,
    AUTHOR = {Ajtai, Mikl\'os and Koml\'os, J\'anos and Szemer\'edi, Endre},
     TITLE = {A note on {R}amsey numbers},
   JOURNAL = {J. Combin. Theory Ser. A},
  FJOURNAL = {Journal of Combinatorial Theory. Series A},
    VOLUME = {29},
      YEAR = {1980},
    NUMBER = {3},
     PAGES = {354--360},
      ISSN = {0097-3165,1096-0899},
   MRCLASS = {05C55 (05C35)},
  MRNUMBER = {600598},
MRREVIEWER = {J.\ E.\ Graver},
       DOI = {10.1016/0097-3165(80)90030-8},
       URL = {https://doi.org/10.1016/0097-3165(80)90030-8},
}

@article {Kim95,
    AUTHOR = {Kim, Jeong Han},
     TITLE = {The {R}amsey number {$R(3,t)$} has order of magnitude
              {$t^2/\log t$}},
   JOURNAL = {Random Structures Algorithms},
  FJOURNAL = {Random Structures \& Algorithms},
    VOLUME = {7},
      YEAR = {1995},
    NUMBER = {3},
     PAGES = {173--207},
      ISSN = {1042-9832,1098-2418},
   MRCLASS = {05C55},
  MRNUMBER = {1369063},
MRREVIEWER = {Pavel\ Valtr},
       DOI = {10.1002/rsa.3240070302},
       URL = {https://doi.org/10.1002/rsa.3240070302},
}

@book {FriezeKarBook,
    AUTHOR = {Frieze, Alan and Karo\'nski, Micha\l},
     TITLE = {Introduction to random graphs},
 PUBLISHER = {Cambridge University Press, Cambridge},
      YEAR = {2016},
     PAGES = {xvii+464},
      ISBN = {978-1-107-11850-8},
   MRCLASS = {05-01 (05C80 60C05 60E15 60F05 60G42)},
  MRNUMBER = {3675279},
       DOI = {10.1017/CBO9781316339831},
       URL = {https://doi.org/10.1017/CBO9781316339831},
}

@article {Ramsey1930,
    AUTHOR = {Ramsey, F. P.},
     TITLE = {On a {P}roblem of {F}ormal {L}ogic},
   JOURNAL = {Proc. London Math. Soc. (2)},
  FJOURNAL = {Proceedings of the London Mathematical Society. Second Series},
    VOLUME = {30},
      YEAR = {1929},
    NUMBER = {4},
     PAGES = {264--286},
      ISSN = {0024-6115},
   MRCLASS = {99-04},
  MRNUMBER = {1576401},
       DOI = {10.1112/plms/s2-30.1.264},
       URL = {https://doi.org/10.1112/plms/s2-30.1.264},
}

@misc{angeltveit2025,
      title={{$R(5,5)\le 46$}}, 
      author={Vigleik Angeltveit and Brendan D. McKay},
      year={2025},
      eprint={2409.15709},
      archivePrefix={arXiv},
      primaryClass={math.CO},
      url={https://arxiv.org/abs/2409.15709}, 
}

@article{campos2025exponentialimprovementdiagonalramsey,
      title={An exponential improvement for diagonal {R}amsey}, 
      author={Marcelo Campos and Simon Griffiths and Robert Morris and Julian Sahasrabudhe},
      journal={Annals of Mathematics},
      volume={to appear},
      year={2025},
      eprint={2303.09521},
      archivePrefix={arXiv},
      primaryClass={math.CO},
      url={https://arxiv.org/abs/2303.09521}, 
}

@article {ES35,
    AUTHOR = {Erd\H{o}s, P. and Szekeres, G.},
     TITLE = {A combinatorial problem in geometry},
   JOURNAL = {Compositio Math.},
  FJOURNAL = {Compositio Mathematica},
    VOLUME = {2},
      YEAR = {1935},
     PAGES = {463--470},
      ISSN = {0010-437X,1570-5846},
   MRCLASS = {99-04},
  MRNUMBER = {1556929},
       URL = {http://www.numdam.org/item?id=CM_1935__2__463_0},
}

@article {Erdos47,
    AUTHOR = {Erd\H{o}s, P.},
     TITLE = {Some remarks on the theory of graphs},
   JOURNAL = {Bull. Amer. Math. Soc.},
  FJOURNAL = {Bulletin of the American Mathematical Society},
    VOLUME = {53},
      YEAR = {1947},
     PAGES = {292--294},
      ISSN = {0002-9904},
   MRCLASS = {56.0X},
  MRNUMBER = {19911},
MRREVIEWER = {H.\ S. M. Coxeter},
       DOI = {10.1090/S0002-9904-1947-08785-1},
       URL = {https://doi.org/10.1090/S0002-9904-1947-08785-1},
}

@misc{gupta2024optimizingcgmsupperbound,
      title={Optimizing the {CGMS} upper bound on {R}amsey numbers}, 
      author={Parth Gupta and Ndiame Ndiaye and Sergey Norin and Louis Wei},
      year={2024},
      eprint={2407.19026},
      archivePrefix={arXiv},
      primaryClass={math.CO},
      url={https://arxiv.org/abs/2407.19026}, 
}

@book {SoiferRamsey,
     TITLE = {{R}amsey theory},
    SERIES = {Progress in Mathematics},
    VOLUME = {285},
    EDITOR = {Soifer, Alexander},
      NOTE = {Yesterday, today, and tomorrow,
              Papers from the workshop held at Rutgers University,
              Piscataway, NJ, May 27--29, 2009},
 PUBLISHER = {Birkh\"auser/Springer, New York},
      YEAR = {2011},
     PAGES = {xiv+189},
      ISBN = {978-0-8176-8091-6},
   MRCLASS = {05-06 (05C55 05D10)},
  MRNUMBER = {2760951},
       DOI = {10.1007/978-0-8176-8092-3},
       URL = {https://doi.org/10.1007/978-0-8176-8092-3},
}

@article {Alon05,
    AUTHOR = {Alon, Noga and R\"odl, Vojt\v{e}ch},
     TITLE = {Sharp bounds for some multicolor {R}amsey numbers},
   JOURNAL = {Combinatorica},
  FJOURNAL = {Combinatorica. An International Journal on Combinatorics and
              the Theory of Computing},
    VOLUME = {25},
      YEAR = {2005},
    NUMBER = {2},
     PAGES = {125--141},
      ISSN = {0209-9683,1439-6912},
   MRCLASS = {05C55},
  MRNUMBER = {2127608},
MRREVIEWER = {Stanley\ M.\ Selkow},
       DOI = {10.1007/s00493-005-0011-9},
       URL = {https://doi.org/10.1007/s00493-005-0011-9},
}

@article {Davies18,
    AUTHOR = {Davies, Ewan and Jenssen, Matthew and Perkins, Will and
              Roberts, Barnaby},
     TITLE = {On the average size of independent sets in triangle-free
              graphs},
   JOURNAL = {Proc. Amer. Math. Soc.},
  FJOURNAL = {Proceedings of the American Mathematical Society},
    VOLUME = {146},
      YEAR = {2018},
    NUMBER = {1},
     PAGES = {111--124},
      ISSN = {0002-9939,1088-6826},
   MRCLASS = {05C69 (05D10 05D40)},
  MRNUMBER = {3723125},
MRREVIEWER = {Hao\ Huang},
       DOI = {10.1090/proc/13728},
       URL = {https://doi.org/10.1090/proc/13728},
}

@misc{mubayi2025k4freetriplesystemslarge,
      title={${K}_4^-$-free triple systems without large stars in the complement}, 
      author={Dhruv Mubayi and Nicholas Spanier},
      year={2025},
      eprint={2504.06076},
      archivePrefix={arXiv},
      primaryClass={math.CO},
      url={https://arxiv.org/abs/2504.06076}, 
}

@article {Conlon23,
    AUTHOR = {Conlon, David and Fox, Jacob and He, Xiaoyu and Mubayi, Dhruv
              and Suk, Andrew and Verstra\"ete, Jacques},
     TITLE = {Hypergraph {R}amsey numbers of cliques versus stars},
   JOURNAL = {Random Structures Algorithms},
  FJOURNAL = {Random Structures \& Algorithms},
    VOLUME = {63},
      YEAR = {2023},
    NUMBER = {3},
     PAGES = {610--623},
      ISSN = {1042-9832,1098-2418},
   MRCLASS = {05C55 (05C65 05D40)},
  MRNUMBER = {4640037},
MRREVIEWER = {Ta\ Sheng\ Tan},
       DOI = {10.1002/rsa.21155},
       URL = {https://doi.org/10.1002/rsa.21155},
}

@misc{campos2025polynomialimprovementoddcyclecomplete,
      title={A polynomial improvement for the odd cycle-complete Ramsey numbers}, 
      author={Marcelo Campos and Matthew Jenssen and Marcus Michelen and Florian Pfender and Julian Sahasrabudhe},
      year={2025},
      eprint={2511.10641},
      archivePrefix={arXiv},
      primaryClass={math.CO},
      url={https://arxiv.org/abs/2511.10641}, 
}

@misc{kühn2025disproofoddhadwigerconjecture,
      title={Disproof of the Odd {H}adwiger Conjecture}, 
      author={Marcus Kühn and Lisa Sauermann and Raphael Steiner and Yuval Wigderson},
      year={2025},
      eprint={2512.20392},
      archivePrefix={arXiv},
      primaryClass={math.CO},
      url={https://arxiv.org/abs/2512.20392}, 
}

\appendix
\section{Inequalities}
Here we briefly record a few standard inequalities we use, in the notation of Frieze and Karo\'nski~\cite{FriezeKarBook}.

\begin{lemma}
\label{binIE}
\begin{align*}
    1+x&\le e^x,~&\forall x.\\
    \left(\frac{n}{k}\right)^k&\le \binom{n}{k}\le \left(\frac{en}{k}\right)^k,~&\forall n\ge k.
\end{align*}
\end{lemma}

\begin{theorem}[Chernoff/Hoeffding Inequality]~\\ \label{CH}
Suppose that
$S_n = X_1 +X_2 +\ldots+X_n$ where $0\le X_i\le 1$ are independent random variables, $\mu_i=\E (X_i)$, $\mu=\sum \mu_i$. Then for $t\ge 0$,
\[
\P\left( S_n\ge \mu+t\right)\le \exp\left( -\frac{t^2}{2(\mu+t/3)}\right),
\]
and for $0\le t\le \mu$
\[
\P\left( S_n\le \mu-t\right)\le \exp\left(-\frac{t^2}{2(\mu-t/3)}\right).
\]
\end{theorem}
\begin{corollary}\label{CHc}
    If $t \geq \frac{3}{2} \mu$, then 
    \[
    \P(S_n\ge \mu + t)\le \exp\left( -\frac{t}{2}\right).
    \]
\end{corollary}
\begin{theorem}[McDiarmid's Inequality]~\\ \label{McD}
Suppose that $Z=Z(W_1,W_2,\ldots,W_n)$ is a random variable that depends on $n$ independent random variables $W_1,W_2,\ldots ,W_n$.
Further, suppose that 
\[
|Z(W_1,W_2,\ldots,W_i,\ldots,W_n)-Z(W_1,W_2,\ldots,W_i',\ldots,W_n)|\le c_i
\]
for all $i=1,2,\ldots, n$ and $W_1,W_2,\ldots,W_n,W_i'$. Then for all $t>0$ we have
\begin{align*}
    \P( Z\ge \E (Z)+t)&\le \exp\left(-\frac{t^2}{2\sum_{i=1}^n c_i^2}\right),\text{ and}\\
    \P( Z\le \E (Z)-t)&\le \exp\left(-\frac{t^2}{2\sum_{i=1}^n c_i^2}\right).
\end{align*}
\end{theorem}

\end{document}